\documentclass{amsart}
\usepackage{amsmath}
\usepackage{amssymb}
\usepackage{euscript}
\usepackage{pstricks}
\newcommand{\Z}{\mathbb{Z}}
\newcommand{\R}{\mathbb{R}}
\newcommand{\GL}{\hbox{\rm GL}}
\newcommand{\tildeA}{\tilde{A}} 
\newcommand{\tildeB}{\tilde{B}} 
\newcommand{\tildeC}{\tilde{C}} 
\newcommand{\tildeD}{\tilde{D}} 

\let\cong\equiv
\newcommand{\sset}{\subseteq}
\newcommand{\generatedby}[1]{\langle #1\rangle}
\newcommand{\biarrow}{\leftrightarrow}
\newcommand{\homotopic}{\simeq}
\newcommand{\semidirect}{\rtimes}

\newcommand{\WOmega}{W_{\Omega}}
\newcommand{\chamber}{K}
\newcommand{\COmega}{\chamber_{\Omega}}
\newcommand{\GOmega}{\Gamma_\Omega}
\newcommand{\DOmega}{\Delta_\Omega}
\newcommand{\Dodd}{\Delta^{\rm odd}} 
\newcommand{\tildeDodd}{\tilde{\Delta}^{\rm odd}} 

\newcommand{\halfV}{\frac{1}{2}V}  

\newtheorem{theorem}{Theorem}
\newtheorem{lemma}[theorem]{Lemma}
\newtheorem{corollary}[theorem]{Corollary}
\newtheorem{proposition}[theorem]{Proposition}
\theoremstyle{remark}

\newtheorem*{remarks}{Remarks}

\newcommand{\tablescaleunit}{.008cm}
\newcommand{\tablescalenoderadius}{8}
\newcommand{\tablescalesinglewidth}{5}
\newcommand{\bigscaleunit}{.012cm}
\newcommand{\bigscalenoderadius}{8}
\newcommand{\bigscalesinglewidth}{5}
\newcommand{\Escaleunit}{.006cm}
\newcommand{\Escalenoderadius}{8}
\newcommand{\Escalesinglewidth}{5}

\newcommand{\solid}[2]{\qdisk(#1,#2){\noderadius}}

\newcommand{\single}[4]{\psline[linewidth=\singlewidth](#1,#2)(#3,#4)}
\newcommand{\dashed}[4]{\psline[linewidth=\singlewidth,linestyle=dashed,dash=10 10](#1,#2)(#3,#4)}
\newcount\triplewidth
\newcommand{\double}[4]{%
\triplewidth=\singlewidth
\multiply\triplewidth by 3
\psline[linewidth=\triplewidth](#1,#2)(#3,#4)
\psline[linewidth=\singlewidth,linecolor=white](#1,#2)(#3,#4)}
\newcommand{\arrowprep}{\psset{linewidth=20,linecolor=lightgray,linearc=95,arrowsize=0pt 1.7,arrowlength=.6,arrowinset=0.1}}

\newcommand{\NoOpToAvoidShiftingProblem}{\rput[c](\Tx,\Ty){}}

\newcount\Ix
\newcount\Iy
\newcount\Cx
\newcount\Cy
\newcount\Fx
\newcount\Fy
\newcount\Tx
\newcount\Ty
\newcommand{\doarrow}[9]{
\Ix=#3
\multiply\Ix by #8
\divide\Ix by 100
\advance\Ix by #1
\Iy=#4
\multiply\Iy by #8
\divide\Iy by 100
\advance\Iy by #2
\Cx=#3
\advance\Cx by #1
\Cy=#4
\advance\Cy by #2
\Fx=#5
\multiply\Fx by -#9
\divide\Fx by 100
\advance\Fx by #5
\advance\Fx by\Cx
\Fy=#6
\multiply\Fy by -#9
\divide\Fy by 100
\advance\Fy by #6
\advance\Fy by\Cy
\psline{#7}(\Ix,\Iy)(\Cx,\Cy)(\Fx,\Fy)
\Tx=#1
\Ty=#2
}
\newcount\tempX
\newcount\tempY
\newcommand{\OutsideWidth}{20}
\newcommand{\InsideWidth}{18}

\newcommand{\PercentAtBeginSingleEnded}{12}
\newcommand{\PercentAtEndSingleEnded}{8}
\newcommand{\PercentAtEndsDoubleEnded}{8}
\newcommand{\PercentAtEndsThirds}{20}
\newcommand{\PercentAtEndsStraight}{15}
\newcommand{\PercentAtBeginStraight}{30}
\newcommand{\PercentAtEndStraight}{25}
\newcommand{\InsidePercentAtBeginSingleEnded}{13}%
\newcommand{\InsidePercentAtEndSingleEnded}{9} %
\newcommand{\InsidePercentAtEndsDoubleEnded}{9}  %
\newcommand{\InsidePercentAtEndsThirds}{21} %
\newcommand{\InsidePercentAtEndsStraight}{17}
\newcommand{\InsidePercentAtBeginStraight}{31}
\newcommand{\InsidePercentAtEndStraight}{28}
\newcommand{\ccwNE}[3]{
\tempX=#1
\advance\tempX by -100
\arrowprep
\psset{linecolor=black,linewidth=\OutsideWidth}
\doarrow{\tempX}{#2}{100}{-100}{100}{100}{->}{\PercentAtBeginSingleEnded}{\PercentAtEndSingleEnded}
\NoOpToAvoidShiftingProblem
\psset{linecolor=lightgray,linewidth=\InsideWidth}
\doarrow{\tempX}{#2}{100}{-100}{100}{100}{->}{\InsidePercentAtBeginSingleEnded}{\InsidePercentAtEndSingleEnded}
\advance\Tx by 100
\advance\Ty by -69
\rput[B](\Tx,\Ty){#3}
}
\newcommand{\ccwSW}[3]{
\tempX=#1
\advance\tempX by 100
\arrowprep
\psset{linecolor=black,linewidth=\OutsideWidth}
\doarrow{\tempX}{#2}{-100}{100}{-100}{-100}{->}{\PercentAtBeginSingleEnded}{\PercentAtEndSingleEnded}
\NoOpToAvoidShiftingProblem
\psset{linecolor=lightgray,linewidth=\InsideWidth}
\doarrow{\tempX}{#2}{-100}{100}{-100}{-100}{->}{\InsidePercentAtBeginSingleEnded}{\InsidePercentAtEndSingleEnded}
\advance\Tx by -100
\advance\Ty by 52
\rput[B](\Tx,\Ty){#3}
}
\newcommand{\cwSE}[3]{
\tempX=#1
\advance\tempX by -100
\arrowprep
\psset{linecolor=black,linewidth=\OutsideWidth}
\doarrow{\tempX}{#2}{100}{100}{100}{-100}{->}{\PercentAtBeginSingleEnded}{\PercentAtEndSingleEnded}
\NoOpToAvoidShiftingProblem
\psset{linecolor=lightgray,linewidth=\InsideWidth}
\doarrow{\tempX}{#2}{100}{100}{100}{-100}{->}{\InsidePercentAtBeginSingleEnded}{\InsidePercentAtEndSingleEnded}
\advance\Tx by 100
\advance\Ty by 52
\rput[B](\Tx,\Ty){#3}
}
\newcommand{\cwNW}[3]{
\tempX=#1
\advance\tempX by 100
\arrowprep
\psset{linecolor=black,linewidth=\OutsideWidth}
\doarrow{\tempX}{#2}{-100}{-100}{-100}{100}{->}{\PercentAtBeginSingleEnded}{\PercentAtEndSingleEnded}
\NoOpToAvoidShiftingProblem
\psset{linecolor=lightgray,linewidth=\InsideWidth}
\doarrow{\tempX}{#2}{-100}{-100}{-100}{100}{->}{\InsidePercentAtBeginSingleEnded}{\InsidePercentAtEndSingleEnded}
\advance\Tx by -100
\advance\Ty by -69
\rput[B](\Tx,\Ty){#3}
}
\newcommand{\dbldown}[3]{
\tempX=#1
\advance\tempX by 100
\arrowprep
\psset{linecolor=black,linewidth=\OutsideWidth}
\doarrow{\tempX}{#2}{-100}{100}{-100}{-100}{<->}{\PercentAtEndsDoubleEnded}{\PercentAtEndsDoubleEnded}
\NoOpToAvoidShiftingProblem
\psset{linecolor=lightgray,linewidth=\InsideWidth}
\doarrow{\tempX}{#2}{-100}{100}{-100}{-100}{<->}{\InsidePercentAtEndsDoubleEnded}{\InsidePercentAtEndsDoubleEnded}
\NoOpToAvoidShiftingProblem
\advance\Tx by -100
\advance\Ty by 52
\rput[B](\Tx,\Ty){#3}
}

\newcommand{\NWSE}[3]{
\arrowprep
\psset{linecolor=black,linewidth=\OutsideWidth}
\doarrow{#1}{#2}{50}{-50}{50}{-50}{<->}{\PercentAtEndsStraight}{\PercentAtEndsStraight}
\NoOpToAvoidShiftingProblem
\psset{linecolor=lightgray,linewidth=\InsideWidth}
\doarrow{#1}{#2}{50}{-50}{50}{-50}{<->}{\InsidePercentAtEndsStraight}{\InsidePercentAtEndsStraight}
\advance\Tx by 50
\advance\Ty by -50
\rput[c](\Tx,\Ty){#3}
}
\newcommand{\SWNE}[3]{
\arrowprep
\psset{linecolor=black,linewidth=\OutsideWidth}
\doarrow{#1}{#2}{50}{50}{50}{50}{<->}{\PercentAtEndsStraight}{\PercentAtEndsStraight}
\NoOpToAvoidShiftingProblem
\psset{linecolor=lightgray,linewidth=\InsideWidth}
\doarrow{#1}{#2}{50}{50}{50}{50}{<->}{\InsidePercentAtEndsStraight}{\InsidePercentAtEndsStraight}
\advance\Tx by 50
\advance\Ty by 50
\rput[c](\Tx,\Ty){#3}
}
\newcommand{\URthird}[3]{
\tempX=#1
\advance\tempX by -50
\tempY=#2
\advance\tempY by 87
\arrowprep
\psset{linecolor=black,linewidth=\OutsideWidth}
\doarrow{\tempX}{\tempY}{100}{0}{50}{-87}{<->}{\PercentAtEndsThirds}{\PercentAtEndsThirds}
\NoOpToAvoidShiftingProblem
\psset{linecolor=lightgray,linewidth=\InsideWidth}
\doarrow{\tempX}{\tempY}{100}{0}{50}{-87}{<->}{\InsidePercentAtEndsThirds}{\InsidePercentAtEndsThirds}
\advance\Tx by 70
\advance\Ty by -13
\rput[Bl](\Tx,\Ty){#3}
}
\newcommand{\LRthird}[3]{
\tempX=#1
\advance\tempX by -50
\tempY=#2
\advance\tempY by -87
\arrowprep
\psset{linecolor=black,linewidth=\OutsideWidth}
\doarrow{\tempX}{\tempY}{100}{0}{50}{87}{<->}{\PercentAtEndsThirds}{\PercentAtEndsThirds}
\NoOpToAvoidShiftingProblem
\psset{linecolor=lightgray,linewidth=\InsideWidth}
\doarrow{\tempX}{\tempY}{100}{0}{50}{87}{<->}{\InsidePercentAtEndsThirds}{\InsidePercentAtEndsThirds}
\advance\Tx by 70
\advance\Ty by 0
\rput[Bl](\Tx,\Ty){#3}
}
\newcommand{\Lthird}[3]{
\tempX=#1
\advance\tempX by -50
\tempY=#2
\advance\tempY by -87
\arrowprep
\psset{linecolor=black,linewidth=\OutsideWidth}
\doarrow{\tempX}{\tempY}{-50}{87}{50}{87}{<->}{\PercentAtEndsThirds}{\PercentAtEndsThirds}
\NoOpToAvoidShiftingProblem
\psset{linecolor=lightgray,linewidth=\InsideWidth}
\doarrow{\tempX}{\tempY}{-50}{87}{50}{87}{<->}{\InsidePercentAtEndsThirds}{\InsidePercentAtEndsThirds}
\advance\Tx by -44
\advance\Ty by 87
\rput[l](\Tx,\Ty){#3}
}
\newcommand{\ULthird}[3]{
\tempX=#1
\advance\tempX by 50
\tempY=#2
\advance\tempY by 87
\arrowprep
\psset{linecolor=black,linewidth=\OutsideWidth}
\doarrow{\tempX}{\tempY}{-100}{0}{-50}{-87}{<->}{\PercentAtEndsThirds}{\PercentAtEndsThirds}
\NoOpToAvoidShiftingProblem
\psset{linecolor=lightgray,linewidth=\InsideWidth}
\doarrow{\tempX}{\tempY}{-100}{0}{-50}{-87}{<->}{\InsidePercentAtEndsThirds}{\InsidePercentAtEndsThirds}
\advance\Tx by -70
\advance\Ty by -13
\rput[Br](\Tx,\Ty){#3}
}
\newcommand{\LLthird}[3]{
\tempX=#1
\advance\tempX by 50
\tempY=#2
\advance\tempY by -87
\arrowprep
\psset{linecolor=black,linewidth=\OutsideWidth}
\doarrow{\tempX}{\tempY}{-100}{0}{-50}{87}{<->}{\PercentAtEndsThirds}{\PercentAtEndsThirds}
\NoOpToAvoidShiftingProblem
\psset{linecolor=lightgray,linewidth=\InsideWidth}
\doarrow{\tempX}{\tempY}{-100}{0}{-50}{87}{<->}{\InsidePercentAtEndsThirds}{\InsidePercentAtEndsThirds}
\advance\Tx by -70
\advance\Ty by 0
\rput[Br](\Tx,\Ty){#3}
}
\newcommand{\Rthird}[3]{
\tempX=#1
\advance\tempX by 50
\tempY=#2
\advance\tempY by 87
\arrowprep
\psset{linecolor=black,linewidth=\OutsideWidth}
\doarrow{\tempX}{\tempY}{50}{-87}{-50}{-87}{<->}{\PercentAtEndsThirds}{\PercentAtEndsThirds}
\NoOpToAvoidShiftingProblem
\psset{linecolor=lightgray,linewidth=\InsideWidth}
\doarrow{\tempX}{\tempY}{50}{-87}{-50}{-87}{<->}{\InsidePercentAtEndsThirds}{\InsidePercentAtEndsThirds}
\advance\Tx by 44
\advance\Ty by -87
\rput[r](\Tx,\Ty){#3}
}
\newcommand{\arrowSW}[2]{
\arrowprep
\psset{linecolor=black,linewidth=\OutsideWidth}
\doarrow{#1}{#2}{-25}{-43}{-25}{-44}{->}{\PercentAtBeginStraight}{\PercentAtEndStraight}
\NoOpToAvoidShiftingProblem
\psset{linecolor=lightgray,linewidth=\InsideWidth}
\doarrow{#1}{#2}{-25}{-43}{-25}{-44}{->}{\InsidePercentAtBeginStraight}{\InsidePercentAtEndStraight}
\NoOpToAvoidShiftingProblem
}
\newcommand{\arrowW}[2]{
\arrowprep
\psset{linecolor=black,linewidth=\OutsideWidth}
\doarrow{#1}{#2}{-50}{0}{-50}{0}{->}{\PercentAtBeginStraight}{\PercentAtEndStraight}
\NoOpToAvoidShiftingProblem
\psset{linecolor=lightgray,linewidth=\InsideWidth}
\doarrow{#1}{#2}{-50}{0}{-50}{0}{->}{32}{27}
\NoOpToAvoidShiftingProblem
}
\newcommand{\arrowNE}[2]{
\arrowprep
\psset{linecolor=black,linewidth=\OutsideWidth}
\doarrow{#1}{#2}{25}{43}{25}{44}{->}{\PercentAtBeginStraight}{\PercentAtEndStraight}
\NoOpToAvoidShiftingProblem
\psset{linecolor=lightgray,linewidth=\InsideWidth}
\doarrow{#1}{#2}{25}{43}{25}{44}{->}{\InsidePercentAtBeginStraight}{\InsidePercentAtEndStraight}
\NoOpToAvoidShiftingProblem
}

\begin{document}
\title{Reflection centralizers in Coxeter groups}
\author{Daniel Allcock}
\thanks{Partly supported by NSF grants DMS-0600112 and DMS-1101566.}
\address{Department of Mathematics\\University of Texas at
  Austin\\Austin, TX 78713}
\email{allcock@math.utexas.edu}
\urladdr{http://www.math.utexas.edu/\textasciitilde allcock}
\date{June 27, 2013}

\begin{abstract}
We refine Brink's theorem, that the non-reflection part of a reflection
centralizer in a Coxeter group $W$ is a free group.  We give an
explicit set of generators for the centralizer, which is finitely generated when $W$ is.  And we give a method for computing the Coxeter
diagram for its reflection subgroup.  In many cases, our method allows one
to compute centralizers in one's head.
\end{abstract}
\maketitle

\noindent
Brink proved the elegant result that the centralizer of a reflection
in a Coxeter group is the semidirect product of a Coxeter group by a
free group \cite{Brink}.  In fact this free group is the fundamental
group of the component of the ``odd Coxeter diagram'' distinguished by
the conjugacy class of the reflection.  Alekseevski, Michor and
Neretin \cite{AMN} independently gave another approach to reflection
centralizers.  We will give several refinements to both papers.

The first refinement is an explicit finite set of generators for the
reflection centralizer; Brink only gave explicit generators for the
free part.  This generating set plays a key role in the author's work
\cite{Allcock-Steinberg-groups} on Steinberg and Kac-Moody groups.  

The second refinement is a method of computing the Coxeter diagram of
the reflection subgroup of the centralizer.  With a little effort we
develop this method to the point that many centralizer computations
are very easy.  For example, the fact that the reflection centralizer
in $W(E_8)$ is $W(E_7)\times2$ becomes a quick mental
computation.  We offer many other examples, 
including the reflection
centralizer 
when the Dynkin
diagram is any cycle of odd edges.  Our most complicated example
is the reflection centralizer in Bugaenko's Coxeter group that acts cocompactly on
$8$-dimensional hyperbolic space \cite{Bugaenko}.  

Our method has some overlap with the Brink-Howlett algorithm for
understanding normalizers of parabolic subgroups in Coxeter groups.  (See
\cite{Brink-Howlett}, and the related \cite{Allcock-normalizers} and
\cite{Borcherds}.)  However, in use it feels quite different.  They
present a certain groupoid, any one of whose maximal subgroups is the
normalizer.

In section~\ref{sec-background} we sketch a proof of Brink's theorem following the
ideas of \cite{AMN}.  This proof is quite different from hers, using
covering spaces and topology in place of induction on word lengths.
We hope this alternate proof will be helpful to some people.
In following sections we give explicit generators for the centralizer,
general rules for computing the Coxeter diagram of its reflection
subgroup, and many examples.

The author is grateful to the Japan Society for the Promotion of
Science, the Clay Mathematics Institute and Kyoto University for their
support and hospitality during this work, to
R.\ Howlett for pointing out a mis-drawn Coxeter diagram, and to one
of the referees for referring me to \cite{AMN}.

\section{Background and previous results}
\label{sec-background}

\noindent
We will review some standard Coxeter group theory, and some results of
Brink, Howlett and Alekseevski-Michor-Neretin.  Our perspective will
be geometric, essentially that of Vinberg from
\cite{Vinberg-discrete-groups-generated-by-linear-reflections}.  We
will sketch the proofs, to unify the original
approaches.

A Coxeter system means a pair $(W,S)$ where $W$ is a group and $S$ is
a set of involutions generating $W$, for which the relations
$(s s')^{o(s s')}=1$ suffice to present $W$, where $s,s'$ vary over $S$ and $o(s s')$ means the order of $s s'$.
A
relation ``$(s s')^\infty=1$'' is regarded as no relation at all.
As usual 
the Coxeter diagram of $(W,S)$, usually written $\Delta$,  means the graph with vertex set
$S$ and $s,s'\in S$ joined by an edge marked $o(s s')$.  When actually
drawing diagrams we follow the standard conventions of omitting edges
that would be labeled $2$, omitting labels from edges that would be
labeled $3$ or $4$, and drawing edges that would be labeled $4$ as
double edges.  

We use the semi-standard term
``spherical'' for a Coxeter system or diagram when the corresponding
group is finite.  This reflects the fact that the group acts naturally
on a sphere, rather than say hyperbolic space.  In the many places
where we refer to the parity of an edge label we use the
convention that $\infty$ is neither even nor odd.

Our first goal is to introduce what we call a Vinberg representation,
which provides the setting for the ideas.
A reflection of a real vector space means a linear transformation
which pointwise fixes a hyperplane (its mirror) and negates some
complementary $1$-dimensional space.  Now suppose $V$ is a
finite-dimensional real vector space, $\Gamma$ is a subgroup of
$\GL(V)$ generated by reflections, and $U$ is an open convex subset of
$V$ which $\Gamma$ preserves and acts on properly discontinuously.
One consequence of proper discontinuity is the local finiteness of the
arrangement of the mirrors of all reflections in $\Gamma$.  It follows that
the complement of the mirrors is open.  By a chamber we mean the
closure in $U$ of a component of this complement.  By local
finiteness, the boundary of a chamber is locally polyhedral, so we may
speak of its faces and their dimensions.  In particular we may speak
of a chamber's facets (codimension-one faces).  Another consequence of
proper discontinuity is that each mirror is the mirror of only one
reflection of $\Gamma$, so we may speak unambiguously of the reflection
across each facet of a chamber.

\begin{theorem}
\label{thm-Vinberg-representations}
Let $\chamber$ be a chamber and $S$ the set of reflections across
its facets.  Then (i) $(\Gamma,S)$ is a Coxeter system, (ii) $\Gamma$
acts freely on the set of chambers, and (iii) every point of $U$ is
$\Gamma$-equivalent to a unique point of $\chamber$.
\end{theorem}

\begin{proof}[Proof sketch]
Theorem~$1$ of \cite[IV.4.4]{Bourbaki} addresses a slightly different
situation, the geometric representation of a Coxeter group.  But the
proof applies verbatim to prove (i) and (ii).  More specifically,
consider the Coxeter system $(W,S)$ where $W$ is the abstract group
with generating set $S$ and defining relations $(s s')^{o(s s')}=1$
where $o(s s')$ means the order of $s s'$ as an element of $\Gamma$.
Then $w(\chamber^\circ)\cap \chamber^\circ=\emptyset$
for all $w\in W-\{1\}$.  It follows that $W\to\Gamma$ is an
isomorphism, hence (i) and (ii).  The proof 
in \cite[IV.4.4]{Bourbaki}
 relies on its Lemma~$1$, which is essentially
a standard-form result for the dihedral group in $\GL(V)$ generated by
any two elements of $S$.  In our case one establishes
this lemma using the definition of a chamber as a component of the
mirror complement, rather than the definition of the geometric
representation.   

For (iii) one follows the proof of Prop.~$5$ of \cite[IV.6]{Bourbaki}.
\end{proof}

In this situation we call $(V,U)$ a Vinberg representation of $\Gamma$
and $(V,U,\chamber)$ a Vinberg representation of $(\Gamma,S)$.  The name
honors Vinberg's proof
\cite{Vinberg-discrete-groups-generated-by-linear-reflections} 
that any Coxeter system $(W,S)$ 
with $S$ finite admits such a representation.
If $S$ is
infinite then a Vinberg representation may still exist, for example by
taking the reflections across the facets of a suitable infinite-sided
polyhedron in hyperbolic space, or by taking a subgroup generated by
reflections in  a Coxeter group that admits a Vinberg representation.  But
it may not, for example
the ascending union of the finite symmetric groups is a Coxeter group
but its only finite-dimensional representation is the trivial one.  
One could
contemplate infinite-dimensional Vinberg representations, at the cost
of more care with topological concepts like proper discontinuity.
In our case
one can deduce results for general Coxeter systems from
the corresponding results for the $|S|<\infty$ case; see corollary~\ref{cor-main-theorem}.

For the rest of this section we will fix a Coxeter system $(W,S)$ that
admits some
Vinberg representation $(V,U,\chamber)$.  Our goal is to understand the
$W$-centralizer of a reflection~$s$, meaning an element conjugate
into $S$.  Although not all Coxeter systems
admit Vinberg representations, most interesting ones do, and it turns
out that understanding this case well enough allows us to later remove
the assumption that such a representation exists.  
Define $\WOmega$ as the subgroup of $W$ generated
by the reflections in $C_W(s)$ other than $s$.  These reflections are
the same as those 
that preserve $s$'s mirror $V^s$ and each of the
two half-spaces it bounds (call one of them $\halfV$).  It is easy to
check that
$(V^s,U^s)$ is a Vinberg representation of $\WOmega$.  By
theorem~\ref{thm-Vinberg-representations}, $\WOmega$ is a Coxeter group.  Now choose a chamber
$\COmega$ for $\WOmega$ in $U^s$ and define $\GOmega$ as the subgroup
of $W$ that preserves $\COmega$ and $\halfV$.  Note that
$\COmega$ has one dimension less than $\chamber$.  To avoid confusion we
specify: when we speak of chambers without mentioning $\WOmega$
explicitly, we always mean chambers of $W$.  We have just sketched a
proof of the
following theorem of Howlett.  (The semidirect product decomposition
comes from the freeness of $\WOmega$'s action on its set of chambers.)

\begin{theorem}[{\cite[corollaries 3 and~7]{Howlett}},\cite{Brink-Howlett}]
\label{thm-Howlett-decomposition}
$C_W(s)=\generatedby{s}\times\generatedby{\WOmega,\GOmega}$, and
the latter factor splits as the semidirect product of $\WOmega$ by
$\GOmega$. 
\qed
\end{theorem}

To understand $C_W(s)$ it now suffices to understand $\WOmega$,
$\GOmega$ and the latter's action on the former.  We begin with
$\GOmega$.  The key to understanding it is that the interior
$\COmega^\circ$ of $\COmega$ turns out to be the universal cover of
part of the boundary of the chamber $\chamber$ that we started with.
A nice mental image of this is to fold
$\COmega^\circ$ along its intersections with mirrors of $W$, and then
wrap it around $\chamber$ as one might wrap a
Weyl-chamber-shaped gift.
This idea is due to Alekseevski, Michor and Neretin \cite{AMN},
and we will sketch it in our language.  Vinberg has
informed me that he and O.~Shvartsman knew of it earlier.

Given a codimension~$2$ face of a chamber, we say the angle there is $\pi/n$
if the product of the reflections corresponding to its two facets has
order $n$.  We define $X$ as the boundary of $\chamber$ in $U$, minus those
codimension~$2$ faces with angle $\pi/(\hbox{even})$.

\begin{proposition}[{\cite[Prop. 2.9]{AMN}}]
\label{prop-open-chamber-is-universal-cover}
The natural map $U\to U/W=\chamber$ induces a universal covering map from
$\COmega^\circ$ to a component of $X$, with deck group $\GOmega$.
\end{proposition}

\begin{proof}[Proof sketch]
The first ingredient is that $\COmega^\circ$ contains no
codimension~$3$ face of any of $W$'s chambers.  (Otherwise, the face's
$W$-stabilizer would be a rank~$3$ spherical Coxeter group, containing
$s$.  Every reflection in such a group centralizes some other
reflection in it.  The face therefore lies in the mirror of a
reflection in $\WOmega$, while $\COmega^\circ$ is disjoint from such
mirrors by its definition.)  Similarly, 
$X$ contains no codimension~$3$ face of $\chamber$.  (Such a face would
correspond to a rank~$3$ spherical Coxeter subsystem of $S$, so some
pair of the three facets involved would make angle
$\pi/2$, so the face lies in a codimension~$2$ face of $\chamber$
which we discarded when defining $X$.)

The same argument shows that $\COmega^\circ$ contains no
codimension~$2$ face of any of $W$'s chambers whose angle is
$\pi/(\hbox{even})$.  And $X$ contains no codimension~$2$ face of $\chamber$ with angle
$\pi/(\hbox{even})$ by definition.

On the other hand, if $\chamber'$ is a chamber of $W$ with a facet $F$ in
$\COmega$, and $f$ is a facet of $F$ whose angle in $\chamber'$ is
$\pi/(\hbox{odd})$ then the interior of $f$ {\it does\/} lie in
$\COmega^\circ$.  This is because $f$'s $W$-stabilizer is a
dihedral group of twice odd order, and in such a group no reflection
centralizes any other.  
Furthermore, the facet (of some other chamber) in $U^s$ on the other
side of $f$ is equivalent under this dihedral group to the other
facet of $\chamber'$ containing $f$ (i.e., not $F$).  
For a picture see \cite[fig.~3]{AMN}.
This is what we referred to
when comparing $\COmega^\circ\to \chamber$ to wrapping a gift.  Using this
and the previous paragraphs one can show that $\COmega^\circ\to X$ is a covering map.  It is a universal
covering of a component of $X$ because $\COmega^\circ$ is convex,
hence connected and simply connected.

That the deck group is $\GOmega$ is the fact that if
$x,y\in\COmega^\circ$ are $W$-equivalent, say $g(x)=y$, then they are
also $\GOmega$-equivalent.  This can be proven by using the fact that
the $W$-stabilizer of any element of $\COmega^\circ$ is either
$\generatedby{s}$ or a dihedral group of twice odd order that contains
$s$.  (And that in such a group all reflections are conjugate.)
\end{proof}

The ``odd Coxeter diagram''
$\Dodd$ means the graph with vertex set $S$ and
$s,s'$ joined just if $o(s s')$ is odd.   
Another way to say this is that $V^s\cap V^{s'}\cap X$ is nonempty
just if $s,s'$ are joined in $\Dodd$.  
Using the fact that $X$ contains no
codimension~$3$ faces of $\chamber$, one can regard
$\Dodd$ as a
deformation-retract of $X$ (cf.\ \cite[Lemma 2.8]{AMN}).

It is well-known that the components of $\Dodd$ correspond to the
conjugacy classes of reflections in $W$.  (If elements of $S$ are oddly joined
then they are conjugate in the group they generate, and if they are
not joined by a chain of odd edges then they
map to distinct elements of $W$'s abelianization.)  So our reflection
$s\in W$ distinguishes a component of $\Dodd$, for which we write
$\Dodd_s$,  and the corresponding
component of $X$.  This leads to the following special case of Brink's
theorem on reflection centralizers.  We will recover her full result,
which does not need a Vinberg representation, as corollary~\ref{cor-Brink}.

\begin{theorem}[Brink \cite{Brink}]
\label{thm-Brink-assuming-Vinberg-representation}
The non-reflection part $\GOmega$ of the
centralizer $C_W(s)$ is the free group $\pi_1(\Dodd_s)$.
\end{theorem}

\begin{proof}
Because $s\in S$, $U^s$ contains
$s$'s facet of $\chamber$.  By choice of the chamber $\COmega$ we may
suppose it does as well.
So the
component $X_s$ of $X$ of which $\COmega^\circ$ is the universal cover
is the one corresponding to the component of $\Dodd$ containing $s$.
Proposition~\ref{prop-open-chamber-is-universal-cover} shows that $\GOmega$ is $\pi_1(X_s)$, 
and the homotopy-equivalence $X\homotopic\Dodd$ identifies
$\pi_1(X_s)$ with
$\pi_1(\Dodd,s)$.
\end{proof}

\section{Explicit generators for the centralizer}
\label{sec-explicit-generators}

\noindent
In corollary~\ref{cor-main-theorem} we give an explicit generating set for a reflection
centralizer in any Coxeter group $W$.  We need this for a forthcoming
application to Steinberg and Kac-Moody groups \cite{Allcock-Steinberg-groups}, and it allows
us to prove the full version of Brink's theorem (corollary~\ref{cor-Brink}).  
Also, if $W$ is finitely generated then so is every reflection
centralizer (corollary~\ref{cor-centralizers-finitely-generated}).

Suppose $(W,S)$ is a Coxeter system with diagram $\Delta$ and
$\gamma=(t_0,\discretionary{}{}{}\dots,\discretionary{}{}{}t_n)$ is an
edge-path in $\Dodd$, with $2l_i+1$ being the
label on the edge joining $t_{i-1}$ and $t_i$.  Then we set
$$
p_\gamma:=(t_1t_0)^{l_1}(t_2t_1)^{l_2}\cdots(t_n t_{n-1})^{l_n}
$$
(or $p_\gamma=1$ if $\gamma$ has length~$0$).  This word is Brink's
$\pi(t_0,\dots,t_n)$.   If $\gamma_2$ starts where $\gamma_1$
ends, then obviously $p_{\gamma_1\gamma_2}=p_{\gamma_1}p_{\gamma_2}$.
If $u$ is a vertex of $\Delta$ evenly joined to  $t_n$, say with edge label $2\lambda$, then we define
$$
r_{\gamma,u}:=p_\gamma\cdot (u t_n)^{\lambda-1}u\cdot p_\gamma^{-1}.
$$
Whenever we refer to $r_{\gamma,u}$ we implicitly require $u$ to be
evenly joined to the endpoint of $\gamma$.
If $\gamma_2$ starts where $\gamma_1$ ends, then obviously
$p_{\gamma_1}r_{\gamma_2,u}p_{\gamma_1}^{-1}=r_{\gamma_1\gamma_2,u}$.

\begin{theorem}
\label{thm-explicit-generators}
Suppose $(W,S)$ admits a Vinberg representation $(V,U,\chamber)$ and $s\in S$.
\begin{enumerate}
\item
\label{item-generators-for-W-Omega}
The set of elements of $W$ of the form $r_{\gamma,u}$, where $\gamma$
is an edge-path starting
at $s$, forms a Coxeter system for $\WOmega$.
\item
\label{item-generators-for-G-Omega}
The map $\pi_1(\Dodd,s)\to W$ given by $\gamma\mapsto p_\gamma$ is an
isomorphism onto $\GOmega$.
\end{enumerate}
\end{theorem}

\begin{remarks}
(1) Once we have proven this theorem we will reprove it as the following
  corollary without assuming the existence of a Vinberg representation.
(2) It may happen that $r_{\gamma,u}=r_{\gamma',u'}$ even if
  $(\gamma,u)\neq(\gamma',u')$.  For this reason the Coxeter system
  consists of the set of elements of $W$ having this form, rather than the
  set of words themselves.  We will work out the equalities of this sort in
  the next section.  
(3) This Coxeter system is the one associated to the
  chamber of $\COmega$ that contains $s$'s facet of~$\chamber$.
\end{remarks}

\begin{corollary}
\label{cor-main-theorem}
Suppose $(W,S)$ is any Coxeter system and $s\in S$.  Define $\WOmega$
as the subgroup of $C_W(s)$ generated by all the reflections it
contains except $s$.  Define $\GOmega$ as the subgroup of $W$ generated by the
elements $p_\gamma$ from \ref{item-generators-for-G-Omega} of
theorem~\ref{thm-explicit-generators}.  Then the conclusions of that theorem hold, and
$C_W(s)=\generatedby{s}\times(\WOmega\semidirect\GOmega)$.
\end{corollary}

\begin{proof}
The corollary is a union of assertions about the subgroups
$\generatedby{S^0}$ where $S^0$ varies over all finite subsets of $S$
containing $s$.
For example, the assertion that $\Sigma:=\{r_{\gamma,u}\}\sset W$ is a Coxeter system
for $\WOmega$ says that a certain map to $\WOmega$ from a Coxeter
group $W_\Sigma$ with generating set $\Sigma$ 
is an isomorphism.   To show surjectivity, consider
a reflection in $C_W(s)$ other than $s$.  It lies in some
$\generatedby{S^0}$, hence in the corresponding subgroup $\WOmega^0$
of $C_{\generatedby{S^0}}(s)$.  Applying theorem~\ref{thm-explicit-generators} shows that it
lies in the group generated by $\Sigma\cap\generatedby{S^0}$.  Similarly, if
$W_\Sigma\to\WOmega$ failed to be injective, then there would be some
finite subset $\Sigma^0\sset\Sigma$ such that
$W_{\Sigma^0}\to\WOmega$ failed to be injective.  But 
$\Sigma^0$ would lie in $\generatedby{S^0}$ for some finite $S^0\sset S$, and the failure of injectivity would
contradict theorem~\ref{thm-explicit-generators}.  
The rest of our claims follow by similar arguments.
\end{proof}

\begin{corollary}[Brink \cite{Brink}]
\label{cor-Brink}
Suppose $(W,S)$ is any Coxeter system and $s\in S$.  Then
the centralizer $C_W(s)$ is the semidirect product of its reflection subgroup by the free group $\pi_1(\Dodd,s)$.
\qed
\end{corollary}

\begin{corollary}
\label{cor-centralizers-finitely-generated}
Suppose $(W,S)$ is any Coxeter system, $s\in S$ and $Z$ is a set of edge-loops
in $\Dodd_s$ generating $\pi_1(\Dodd,s)$.  Suppose given edge-paths
$\delta_t$ in $\Dodd_s$ from $s$ to $t$, for each $t\in\Dodd_s$.  Then
the $p_{\gamma\in Z}$ generate $\GOmega$ and together with the
$r_{\delta_t,u}$ they generate $\WOmega\semidirect\GOmega$.  In
particular, if $S$ is finite then $C_W(s)$ is finitely generated.
\qed
\end{corollary}

The rest of this section is devoted to proving theorem~\ref{thm-explicit-generators}.  The
focus in the covering space argument in proposition~\ref{prop-open-chamber-is-universal-cover} was on
$\COmega^\circ$, but now we focus on $\COmega$ to avoid fussing over
the missing faces.  By a \emph{tile} we mean a facet (of some
$W$-chamber) that lies in $\COmega$.  Its \emph{type} means its image
in $\chamber=U/W$, which is a facet of $\chamber$.  
To avoid confusion we write $F_t$ for the facet of $\chamber$ corresponding
to $t\in S$.
By proposition~\ref{prop-open-chamber-is-universal-cover} the tiles correspond to the nodes of the
universal cover $\tildeDodd_s$, and their types to the nodes of
$\Dodd_s$.  By the \emph{base tile} we mean the tile $\chamber\cap\COmega$,
for which we already have the name $F_s$.
Now we can explain the words $p_\gamma$:

\begin{lemma}
\label{lem-role-of-the-p-gammas}
Suppose $T$ is a tile, $(T_0=F_s,T_1,\dots,T_n=T)$ is an edge-path in
$\tildeDodd_s$ from the base tile to it, and $\gamma=(t_0=s,\dots,t_n)$ is
its projection to $\Dodd_s$.  Then $p_\gamma$ sends $\chamber$ to the unique
chamber in $\halfV$ having $T$ as a facet, and  it sends 
$F_{t_n}$ to $T_n$.
\end{lemma}

\begin{proof}
The key claim is the following: 
the image of $F_{t_n}$ under
$(t_n t_{n-1})^{l_n}$ lies in $t_{n-1}$'s mirror and its intersection with
$F_{t_{n-1}}$ is
$F_{t_{n-1}}\cap F_{t_n}$; furthermore, $\chamber$ and $(t_n
t_{n-1})^{l_n}(\chamber)$ lie on the same side of this mirror.  The proof is
a picture-drawing exercise in a dihedral group of twice odd order
(cf. \cite[Fig.~3]{AMN}).

The lemma is just this claim wrapped inside an induction.  The case
$n=0$ is trivial, so take $n>0$ and let $\beta$ be the subpath
$(T_0,\dots,T_{n-1})$ of $\gamma$.  The induction hypothesis tells us
that $p_\beta$ sends $F_{t_{n-1}}$ to $T_{n-1}$ and $\chamber$ into $\halfV$.
Conjugating the previous paragraph's claim by $p_\beta$ tells us that
$p_\beta\circ(t_n t_{n-1})^{l_n}\circ p_\beta^{-1}$ has two
properties.  First, it sends $p_\beta(F_{t_n})$ into the mirror
containing $p_\beta(F_{t_{n-1}})=T_{n-1}$ (this mirror is $V^s$) and
on the other side of $p_\beta(F_{t_{n-1}}\cap F_{t_n})$ from
$T_{n-1}$.  Second, it sends $p_\beta(\chamber)$ into the same side of $V^s$
containing $p_\beta(\chamber)$.  These two statements unravel to give
$p_\gamma(F_{t_n})=T_n$ and $p_\gamma(\chamber)\sset\halfV$.
\end{proof}

Now we address $\WOmega$.  Given a tile $T$ of type $t$, $T$'s facets
correspond to intersections of $F_t$ with other facets $F_u$ of $\chamber$.
Given a facet $F_{u\neq t}$ of $\chamber$ that meets $F_t$, we write $[T,u]$
for the corresponding facet of $T$.  If $F_t,F_u$ make angle
$\pi/(\hbox{odd})$, then by proposition~\ref{prop-open-chamber-is-universal-cover} there is another tile
on the other side of $[T,u]$ sharing that facet.  On the other hand, if the angle is
$\pi/(\hbox{even})$ then we know from the same proposition that
$[T,u]$ is not in $\COmega^\circ$, so it lies in a facet of $\COmega$.
That is, $\WOmega$ contains a reflection across $[T,u]$.  Every facet
of $\COmega$ contains some such $[T,u]$.  The discussion after
theorem~\ref{thm-Vinberg-representations} tells us that the reflections obtained this way form a
Coxeter system for $\WOmega$.  These reflections are just the $r_{\gamma,u}$'s:

\begin{lemma}
\label{lem-reflections-in-W-Omega}
Suppose $T$ is a tile of type $t$ and $u\in S$ is evenly joined to
$t$.  Suppose $(T_0=F_s,T_1,\dots,T_n=T)$ is a path in $\tildeDodd$,
and $\gamma$ the corresponding path in $\Dodd$.  Then $r_{\gamma,u}$
lies in $\WOmega$ and is the reflection whose mirror contains $[T,u]$.
\end{lemma}

\begin{proof}
This is similar to lemma~\ref{lem-role-of-the-p-gammas}; recall from the definition of
$r_{\gamma,u}$ that we write $2\lambda$ for the edge label between
$t_n$ and $u$.  The key claim is that $(u t_n)^{\lambda-1}u$ is the
reflection in $\generatedby{t_n,u}$ that centralizes $t_n$ (other than
$t_n$ itself).  This is a slightly different picture-drawing argument
than before.  The lemma follows from this claim just as before.
\end{proof}

\begin{proof}[Proof of theorem~\ref{thm-explicit-generators}]
Before lemma~\ref{lem-reflections-in-W-Omega} we explained how the facets of $\COmega$
are covered by the $[T,u]$, and this lemma tells us that the
$r_{\gamma,u}$'s are the reflections across them.  We know from
theorem~\ref{thm-Vinberg-representations} that these reflections form a Coxeter system for
$\WOmega$, proving part \ref{item-generators-for-W-Omega}.

For part \ref{item-generators-for-G-Omega} we note that an element of
$\GOmega$ sends $F_s$ to some tile, hence equals
$p_\gamma$ for some path $\gamma$ in $\Dodd$ based at $s$.  This
equality uses
the simple transitivity of $W$ on chambers.  If the endpoint of
$\gamma$ in $\Dodd$ is $t$, then lemma~\ref{lem-role-of-the-p-gammas} shows that $p_\gamma$
conjugates $t$ to $s$.  Since membership in $\GOmega$ requires
commutativity with $s$, this shows $t=s$.  So $p_\gamma$ can only lie
in $\GOmega$ if it is a loop.  Conversely, if $\gamma$ is a loop then
the lemma shows that $p_\gamma$ sends $F_s$ into $V^s$ and hence
centralizes~$s$.  This shows that the $p_\gamma$'s with $\gamma$ a
loop generate $\GOmega$, and the freeness of $\GOmega$ comes from
theorem~\ref{thm-Brink-assuming-Vinberg-representation}. 
\end{proof}

\section{The Coxeter diagram of $\WOmega$}
\label{sec-diagram-of-stabilizer}

\noindent
We showed in theorem~\ref{thm-Vinberg-representations} that $\WOmega$ is a Coxeter group, and
described a Coxeter system for it in corollary~\ref{cor-main-theorem}.  Our goal in this
section is to work out the Coxeter diagram $\DOmega$ in a manner
making obvious the action of $\GOmega$.  This
gives a complete answer to the problem of  presenting reflection centralizers in Coxeter groups.  We
will use the geometric language of Vinberg representations, but the
results transfer to general Coxeter groups by the methods used for
corollary~\ref{cor-main-theorem}.

Recall from section~\ref{sec-explicit-generators} that a tile means a facet (of some
$W$-chamber) that lies in $\COmega$.  By an \emph{arrow} we mean a
facet of a tile, that is not a facet of any other tile, i.e., it lies
in the boundary of $\COmega$.  This peculiar terminology helps
organize the calculations in examples; see section~\ref{sec-examples}.  The arrows
fall into equivalence classes according to which facet of $\COmega$
they lie in, which we call the arrow classes.  To describe $\DOmega$
we must find the arrow classes and understand how the corresponding
facets of $\COmega$ meet.  

Whenever we use a symbol with a tilde, such as $\tildeA$, for a node
of $\tildeDodd_s$, we will use the corresponding symbol without the
tilde for its image in $\Dodd$.
We explained in section~\ref{sec-explicit-generators} that the arrows are in bijection with
the pairs $[\tildeA,B]$, where $\tildeA$ and $B$ are vertices of
$\tildeDodd_s$ and $\Delta$ respectively, and the edge joining $A$ and $B$
in $\Delta$ is evenly-labeled (which includes the case that the edge is absent).
The next step is to determine the dihedral angles among the arrows.
(If two arrows meet in codimension~$1$ then we say they make dihedral angle $\pi/n$ as a shorthand for the
product of their reflections having order~$n$.  This order may be~$1$
because they may lie in the same facet of $\COmega$.)

\begin{lemma}
\label{lem-facet-relations}
If $\tildeA$ is a vertex of $\tildeDodd_s$ and $B$ and $C$ are
vertices of $\Delta$, such that the subdiagram of $\Delta$ formed by $A$, $B$
and $C$ appears in table~\ref{tab-of-dihedral-angles}, then the
indicated arrows intersect in codimension~$1$, with the stated
dihedral angle.  Conversely, if two arrows meet in codimension~$1$
then there exist such $\tildeA,B,C$, such that the arrows are the ones
indicated in the table.
\end{lemma}

\psset{unit=\tablescaleunit}
\newcommand{\noderadius}{\tablescalenoderadius}
\newcommand{\singlewidth}{\tablescalesinglewidth}
\begin{table}
\begingroup
\newcommand{\OR}{\quad\hbox{or}\quad}
\newcommand{\BETWEEN}{\noalign{\vskip10pt}}
\newcommand{\noORspacer}{\kern27pt}
\begin{align}
\label{eq-A4B3C}
\raisebox{-.4\height}{%
\begin{pspicture}(0,0)(100,87)
\rput[b](50,100){$A$}%
\rput[r](0,0){$B$\,\,}%
\rput[l](100,0){\,\,$C$}%
\double{50}{87}{100}{0}%
\single{0}{0}{100}{0}%
\solid{50}{87}
\solid{0}{0}
\solid{100}{0}
\end{pspicture}
}%
\kern20pt
&\implies
\hbox{$[\tildeA,B]$ and $[\tildeA,C]$ make angle $\pi/4$}.
\\
\BETWEEN
%
\label{eq-AevenB2C}
\raisebox{-.4\height}{%
\begin{pspicture}(0,0)(100,87)
\single{50}{87}{0}{0}%
\rput[rb](25,43){even\,}%
\solid{50}{87}
\solid{0}{0}
\solid{100}{0}
\end{pspicture}
}%
\kern20pt
&\implies
[\tildeA,B]\mathrel{\bot}[\tildeA,C].
\\
\BETWEEN
\label{eq-A2BnC}
\raisebox{-.4\height}{%
\begin{pspicture}(0,0)(100,87)
\single{0}{0}{100}{0}%
\rput[b](50,10){$n$}%
\solid{50}{87}
\solid{0}{0}
\solid{100}{0}
\end{pspicture}
}%
\kern20pt
&\implies
\hbox{$[\tildeA,B]$ and $[\tildeA,C]$ make angle $\pi/n$}.
\\
\BETWEEN
%
\label{eq-B4A3C}
\raisebox{-.4\height}{%
\begin{pspicture}(0,0)(100,87)
\double{50}{87}{0}{0}%
\single{50}{87}{100}{0}%
\solid{50}{87}
\solid{0}{0}
\solid{100}{0}
\end{pspicture}
}%
\kern20pt
&\implies
[\tildeA,B]\mathrel{\bot}[\tildeC,B].
\\
\BETWEEN
\label{eq-AoddC2B}
\raisebox{-.4\height}{%
\begin{pspicture}(0,0)(100,87)
\single{50}{87}{100}{0}%
\rput[lb](75,43){\rm\,odd}%
\solid{50}{87}
\solid{0}{0}
\solid{100}{0}
\end{pspicture}
}%
\kern20pt
&\implies
\hbox{$[\tildeA,B]$ and $[\tildeC,B]$ make angle $\pi$}.
\\
\BETWEEN
\label{eq-A3C3B}
\raisebox{-.4\height}{%
\begin{pspicture}(0,0)(100,87)
\single{100}{0}{0}{0}%
\single{50}{87}{100}{0}%
\solid{50}{87}
\solid{0}{0}
\solid{100}{0}
\end{pspicture}
}%
\kern20pt
&\implies
\hbox{$[\tildeA,B]$ and $[\tildeB,A]$ make angle $\pi$}.
\\
\BETWEEN
%
\label{eq-A5C3B}
\raisebox{-.4\height}{%
\begin{pspicture}(0,0)(100,87)
\single{100}{0}{0}{0}%
\single{50}{87}{100}{0}%
\rput[bl](75,43){\,$5$}
\solid{50}{87}
\solid{0}{0}
\solid{100}{0}
\end{pspicture}
}%
\kern20pt
&\implies
[\tildeA,B]\mathrel{\bot}[\tildeB,A].
\end{align}
\endgroup
\caption{Dihedral angles between arrows (facets of tiles); see
  lemma~\ref{lem-facet-relations}.  We implicitly label the vertices
  of all the diagrams $A$, $B$, $C$ as in the first one. 
 In the last four lines $\tildeC$ means the tile of
  type $C$ adjacent to $\tildeA$, and in the last 
  two lines $\tildeB$ means the tile of type $B$ adjacent to
  $\tildeC$.}
\label{tab-of-dihedral-angles}
\end{table}

\begin{proof}
Suppose $[\tildeA,B]$ and $[\tildeA',B']$ are arrows, whose
intersection $\phi$ has codimension one in each.  Write $\chamber$ for the
chamber associated to $\tildeA$.  Then $\phi$ is a codimension~$3$
face of $\chamber$, so it corresponds to a spherical $3$-vertex subdiagram of
$\Delta_\chamber$, containing $A$ and $B$.  Write $Y$ for the corresponding finite Coxeter
group and $C$ for the third vertex.  We
next verify the conclusions of the theorem if this subdiagram
appears in table~\ref{tab-of-dihedral-angles}.

The calculation takes place entirely in the standard representation of
$Y$, which we think of as transverse to
$\phi$.  In this $\R^3$, $U^s$ appears as a hyperplane $H$, $\halfV$
as a half-space bounded by $H$, $\chamber$ as a chamber of $Y$ with a facet
in $H$, and $\tildeA$ equal to this facet.  The two other facets of
this chamber correspond to $B$ and $C$.  We write $F_A$, $F_B$ and
$F_C$ for these facets, and $R_{A B}$, $R_{B C}$ and $R_{C A}$ for the
rays where these facets meet.

The simplest case is when $A$ and $C$ are evenly joined.  Then
$[\tildeA,C]$ is also an arrow containing $\phi$,  so it is the only
one other than $[\tildeA,B]$.  Also, $[\tildeA,B]$
and $[\tildeA,C]$ correspond to $R_{A B}$ and $R_{AC}$, so the angle
between them is the angle between these rays in
$\R^3$.  

The next case is when $A$ and $C$ are oddly joined and $C$ and $B$ are
unjoined or evenly joined.  Let $\Theta$ be the rotation around
$R_{C A}$ with $\Theta(F_C)$ in $H$ but not overlapping $F_A$.  Then
$\tildeC$ corresponds to $\Theta(F_C)$, and $[\tildeC,B]$ is the arrow
meeting $[\tildeA,B]$ in $\phi$; it corresponds to $\Theta(R_{B C})$.
So the angle between these arrows is the angle in $\R^3$ between $R_{A
  B}$
and $\Theta(R_{B C})$.  (This rotation process is the
reverse of the gift-wrapping process of section~\ref{sec-background}.)

In the final case, $A$ and $C$ are oddly joined and so are $C$ and
$B$.  We will apply
a second rotation.  Namely, let $\Theta'$ be the rotation around
$\Theta(R_{B C})$ with $\Theta'\circ\Theta(F_B)$ in $H$ but not
overlapping $\Theta(F_C)$.  Then $\tildeB$ corresponds to
$\Theta'\circ\Theta(F_B)$, and $[\tildeB,A]$ is the arrow meeting
$[\tildeA,B]$ in $\phi$; it corresponds to $\Theta'\circ\Theta(R_{A B})$.  The angle between these arrows is the angle
in $\R^3$ between $R_{A B}$ and $\Theta'\circ\Theta(R_{A B})$.  

One can find these angles without computation.  Consider the edges
$E_A$, $E_B$ and $E_C$ of the spherical triangle defined by $F_A$,
$F_B$ and $F_C$.  In the three cases the desired angle is $\ell(E_A)$,
$\ell(E_A)+\ell(E_C)$ and $\ell(E_A)+\ell(E_B)+\ell(E_C)$, where
$\ell$ indicates length.  Drawing the tessellation of the sphere by
$Y$'s chambers makes it easy to recognize which submultiple of $\pi$
this is.  This justifies the entries in table~\ref{tab-of-dihedral-angles}.

The table is complete because one can write down all possibilities
for the spherical diagram on $A$, $B$ and $C$; $A$ and $B$ should be evenly
joined since $[\tildeA,B]$ is an arrow.  The possibilities
other than those in table~\ref{tab-of-dihedral-angles} are
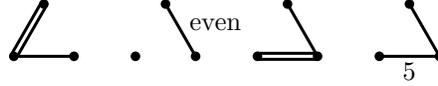
$$
\raisebox{-.4\height}{%
\begin{pspicture}(0,0)(100,87)
\double{50}{87}{0}{0}%
\single{0}{0}{100}{0}%
\solid{50}{87}
\solid{0}{0}
\solid{100}{0}
\end{pspicture}
}%
\qquad
\raisebox{-.4\height}{%
\begin{pspicture}(0,0)(100,87)
\single{50}{87}{100}{0}%
\rput[lb](75,43){\rm\,\,even}%
\solid{50}{87}
\solid{0}{0}
\solid{100}{0}
\end{pspicture}
}%
\qquad
\raisebox{-.4\height}{%
\begin{pspicture}(0,0)(100,87)
\double{100}{0}{0}{0}%
\single{50}{87}{100}{0}%
\solid{50}{87}
\solid{0}{0}
\solid{100}{0}
\end{pspicture}
}%
\qquad
\raisebox{-.4\height}{%
\begin{pspicture}(0,0)(100,87)
\single{100}{0}{0}{0}%
\single{50}{87}{100}{0}%
\rput[t](50,-10){$5$}
\solid{50}{87}
\solid{0}{0}
\solid{100}{0}
\end{pspicture}
}%
$$
\vskip12pt 
\noindent
where we use the table's labeling of vertices.  The first and second differ
from \eqref{eq-A4B3C} and \eqref{eq-AevenB2C} by $B\biarrow C$, the third from \eqref{eq-B4A3C}
by $A\biarrow C$ and the last from \eqref{eq-A5C3B} by 
$A\biarrow B$.  In all cases the conclusion of the relevant line of
the table is symmetric under the same interchange.  So if the diagram
of $A,B,C$ doesn't appear in the table then we just swap 
$[\tildeA,B]$ and $[\tildeA',B']$ .
\end{proof}

The lemma allows one to compute the Coxeter diagram $\DOmega$:

\begin{theorem}
\label{thm-dihedral-angles}
The equivalence relation of arrows lying in the same facet of
$\COmega$ is generated by the equivalence of $[\tildeA,B]$ with
$[\tildeC,B]$ in the situation of \eqref{eq-AoddC2B} and that of
$[\tildeA,B]$ with $[\tildeB,A]$ in the situation of \eqref{eq-A3C3B}.

If two
arrow classes have representative arrows as in one of the other
entries of the table, then the corresponding facets of $\COmega$ have
the listed dihedral angle.  If they have no such representatives then
they do not meet, and the edge of $\DOmega$ joining them is
labeled~$\infty$. 
\end{theorem}

\begin{proof}
If two arrows lie in the same facet of $\COmega$ then there is a chain
of arrows joining them, each lying in that facet of $\COmega$ and
meeting the next in codimension~$1$.  This proves the first claim.
The second follows immediately from lemma~\ref{lem-facet-relations},
and the third is obvious.
\end{proof}

We close the section with a few remarks on the language.  We visualize
$[\tildeA,B]$ as an arrow pointing from the vertex
$\tildeA$ of $\tildeDodd_s$ to the vertex $B$ of $\Delta$.  We define a tail
class as an equivalence class of facets under the relation
\eqref{eq-AoddC2B}, i.e., $[\tildeA,B]$ and $[\tildeC,B]$ are
equivalent when $A,B,C$ form the configuration \eqref{eq-AoddC2B}.
The reason for the name is that the equivalence corresponds the tail moving around in
$\tildeDodd_s$  while the head stays fixed in $\Delta$.
\psset{unit=\bigscaleunit}
\renewcommand{\noderadius}{\bigscalenoderadius}
\renewcommand{\singlewidth}{\bigscalesinglewidth}
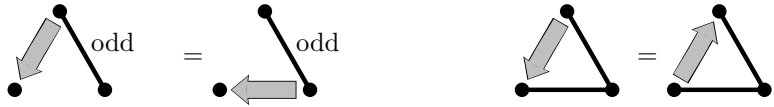
$$
\raisebox{-.4\height}{%
\begin{pspicture}(-10,-10)(110,90)
\solid{0}{0}
\solid{50}{87}
\solid{100}{0}
\single{50}{87}{100}{0}
\rput[bl](75,43){\,\,odd}
\arrowSW{50}{87}
\end{pspicture}
}%
\kern20pt
=
\raisebox{-.4\height}{%
\begin{pspicture}(-10,-10)(110,90)
\solid{0}{0}
\solid{50}{87}
\solid{100}{0}
\single{50}{87}{100}{0}
\rput[bl](75,43){\,\,odd}
\arrowW{100}{0}
\end{pspicture}
}%
\kern10pt
\qquad\qquad\qquad
\raisebox{-.4\height}{%
\begin{pspicture}(-10,-10)(110,90)
\solid{0}{0}
\solid{50}{87}
\solid{100}{0}
\single{50}{87}{100}{0}
\single{0}{0}{100}{0}
\arrowSW{50}{87}
\end{pspicture}
}%
=
\raisebox{-.4\height}{%
\begin{pspicture}(-10,-10)(110,90)
\solid{0}{0}
\solid{50}{87}
\solid{100}{0}
\single{50}{87}{100}{0}
\single{0}{0}{100}{0}
\arrowNE{0}{0}
\end{pspicture}
}%
$$ The second picture is a graphical interpretation of
\eqref{eq-A3C3B}, although one must be careful keeping track of which
vertices lie in $\tildeDodd_s$ and which lie in $\Delta$.  In most of our examples in
the next section $\Dodd_s$ will be a tree, so $\tildeDodd_s$ may be
regarded as a subdiagram of $\Delta$.  Then we can take the figure
literally.

\section{Examples}
\label{sec-examples}

\noindent
In this section we give many examples of reflection centralizers in
Coxeter groups, illustrating
theorem~\ref{thm-dihedral-angles}.  
See \cite{AMN} for some different and very nice examples worked from another perspective.
We generally proceed by working out the tail classes, fusing
them into arrow classes, and then finding the angles.  

Suppose first that $\Delta$ is a tree of single edges (edge label~$3$).
Then there is only one class of reflection, so we don't need to choose
a component of $\Dodd$.  The tail classes are easy to work out: each
contains a unique arrow $[A,B]$ where $A$ and $B$ have distance~$2$ in
$\Delta$.  (Proof: move the tail toward the head.)  The tail classes fuse
in pairs, got by reversing these arrows.  The end result is that the
generators for $\WOmega$ are in bijection with the $A_3$ diagrams in
$\Delta$.  Almost all the angles can be worked out using \eqref{eq-A2BnC}.
In our situation it reads: if two arrow classes have representatives
with the same tail, then their edge label in $\DOmega$ is the same as
the one between their tips in $\Delta$.

The first example is $A_{n\geq3}$, which has $n-2$ arrow classes:
$$
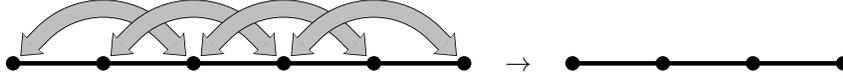
\begin{pspicture*}(-\noderadius,-\noderadius)(508,82)
\solid{0}{0}
\solid{100}{0}
\solid{200}{0}
\solid{300}{0}
\solid{400}{0}
\solid{500}{0}
\single{0}{0}{500}{0}
\dbldown{100}{0}{}
\dbldown{200}{0}{}
\dbldown{300}{0}{}
\dbldown{400}{0}{}
\end{pspicture*}
\quad\to\quad
\begin{pspicture*}(-\noderadius,-\noderadius)(308,82)
\solid{0}{0}
\solid{100}{0}
\solid{200}{0}
\solid{300}{0}
\single{0}{0}{300}{0}
\end{pspicture*}
$$ We have drawn double-headed arrows because of the fusion of
tail classes.  Every arrow class has a representative with
tail at the leftmost vertex.  So the joins between these arrow classes
are the same as the joins between the right-hand tips
of the arrows.  So $\WOmega=W(A_{n-2})$.  

The second example is $D_{n\geq6}$, which has $n-1$ arrow classes:
$$
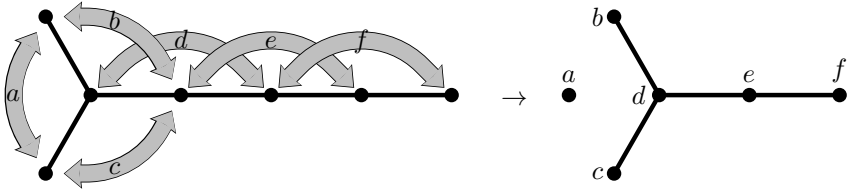
\raisebox{-.45\height}{%
\begin{pspicture*}(-100,-104)(408,104)
\solid{0}{0}
\solid{-50}{87}
\solid{-50}{-87}
\solid{100}{0}
\solid{200}{0}
\solid{300}{0}
\solid{400}{0}
\single{0}{0}{400}{0}
\single{0}{0}{-50}{87}
\single{0}{0}{-50}{-87}
\dbldown{100}{0}{$d$}
\dbldown{200}{0}{$e$}
\dbldown{300}{0}{$f$}
\Lthird{0}{0}{$a$}
\URthird{0}{0}{$b$}
\LRthird{0}{0}{$c$}
\end{pspicture*}
}%
\quad\to\quad
\raisebox{-.45\height}{%
\begin{pspicture}(-108,-104)(208,104)
\solid{0}{0}
\solid{-50}{87}
\solid{-50}{-87}
\solid{100}{0}
\solid{200}{0}
\single{0}{0}{200}{0}
\single{0}{0}{-50}{87}
\single{0}{0}{-50}{-87}
\rput[r](-62,87){$b$}
\rput[r](-62,-87){$c$}
\rput[r](-15,0){$d$}
\rput[b](100,14){$e$}
\rput[b](200,14){$f$}
\solid{-100}{0}
\rput[b](-100,14){$a$}
\end{pspicture}
}%
$$
Choosing representative arrows with tails at the top left shows that
$a$ is orthogonal to all the other generators except perhaps $c$.
Repeating the argument with tails at the lower left shows that $a$ is
also orthogonal to $c$.  Then taking tails at the rightmost vertex shows
that $b,\dots,f$ form a $D_{n-2}$ diagram.  So $\WOmega=W(A_1D_{n-2})$.
The  $D_4$ and $D_5$ cases are the same provided one interprets $D_2$
and $D_3$ as $A_1^2$ and $A_3$.   

The third example is the affine
diagram $\tildeD_{n\geq6}$, which 
gives $\WOmega=W(\tildeA_1\tildeD_{n-2})$ in a similar way:
$$
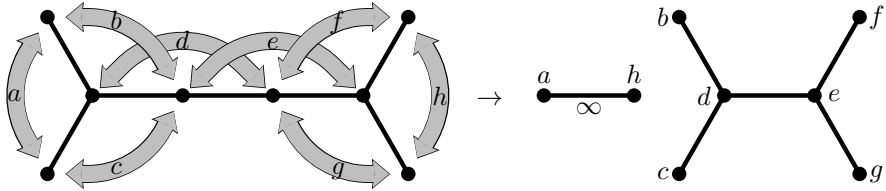
\raisebox{-.45\height}{%
\begin{pspicture}(-100,-104)(408,104)
\solid{0}{0}
\solid{-50}{87}
\solid{-50}{-87}
\solid{100}{0}
\solid{200}{0}
\solid{300}{0}
\solid{350}{87}
\solid{350}{-87}
\single{0}{0}{300}{0}
\single{0}{0}{-50}{87}
\single{0}{0}{-50}{-87}
\single{300}{0}{350}{87}
\single{300}{0}{350}{-87}
\dbldown{100}{0}{$d$}
\dbldown{200}{0}{$e$}
\Lthird{0}{0}{$a$}
\URthird{0}{0}{$b$}
\LRthird{0}{0}{$c$}
\Rthird{300}{0}{$h$}
\ULthird{300}{0}{$f$}
\LLthird{300}{0}{$g$}
\end{pspicture}
}%
\to\quad
\raisebox{-.45\height}{%
\begin{pspicture}(-208,-104)(158,104)
\solid{0}{0}
\solid{-50}{87}
\solid{-50}{-87}
\solid{100}{0}
\solid{150}{87}
\solid{150}{-87}
\single{0}{0}{100}{0}
\single{0}{0}{-50}{87}
\single{0}{0}{-50}{-87}
\single{100}{0}{150}{87}
\single{100}{0}{150}{-87}
\rput[r](-62,87){$b$}
\rput[r](-62,-87){$c$}
\rput[r](-15,0){$d$}
\rput[l](160,87){$f$}
\rput[l](161,-87){$g$}
\rput[l](115,0){$e$}
\solid{-200}{0}
\solid{-100}{0}
\single{-200}{0}{-100}{0}
\rput[t](-150,-8){$\infty$}
\rput[b](-200,14){$a$}
\rput[b](-100,14){$h$}
\end{pspicture}
}%
$$
The new phenomenon is that the arrows $a$ and $h$ cannot be moved into
a spherical $3$-vertex diagram.  
That is, their arrow classes contain no representatives lying in such
a diagram.  (In fact each arrow is its entire class.)
So those facets of
$\COmega$ don't meet, hence the edge label
$\infty$.  The $\tildeD_4$ and $\tildeD_5$ cases are the same,
provided one interprets $\tildeD_2$ and $\tildeD_3$ as $\tildeA_1^2$
and $\tildeA_3$.

These examples are enough to treat the general case:

\begin{theorem}
\label{thm-centralizer-diagram-tree-case}
If $\Delta$ is a tree of single edges, then the vertices of $\DOmega$ are the
$A_3$ subdiagrams of $\Delta$, with edge labels as follows.  If the
convex hull of two $A_3$'s in $\Delta$ has type $D$ (resp.\ $\tildeD$), then their edge label in $\DOmega$ is $2$ (resp. $\infty$).
Otherwise, it is the same as the one in $\Delta$ between their middle
vertices.  \qed
\end{theorem}

In the special case of trivalent branch points, no two adjacent,
$\DOmega$ can be got from $\Delta$ by the following operation: ``blow up''
each branch point
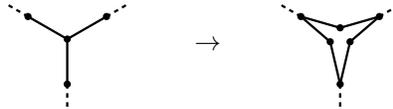
\psset{unit=\Escaleunit}
\renewcommand{\noderadius}{\Escalenoderadius}
\renewcommand{\singlewidth}{\Escalesinglewidth}
$$
\raisebox{-.5\height}{%
\begin{pspicture}(-130,-150)(130,100)
\solid{-87}{50}
\solid{87}{50}
\solid{0}{-100}
\solid{0}{0}
\single{0}{0}{-87}{50}
\single{0}{0}{87}{50}
\single{0}{0}{0}{-100}
\dashed{-87}{50}{-130}{75}
\dashed{87}{50}{130}{75}
\dashed{0}{-100}{0}{-150}
\end{pspicture}
}%
\qquad\to\qquad
\raisebox{-.5\height}{%
\begin{pspicture}(-130,-150)(130,100)
\solid{-87}{50}
\solid{87}{50}
\solid{0}{-100}
\solid{-22}{-6}
\solid{22}{-6}
\solid{0}{25}
\single{0}{25}{-87}{50}
\single{0}{25}{87}{50}
\single{22}{-6}{87}{50}
\single{22}{-6}{0}{-100}
\single{-22}{-6}{0}{-100}
\single{-22}{-6}{-87}{50}
\dashed{-87}{50}{-130}{75}
\dashed{87}{50}{130}{75}
\dashed{0}{-100}{0}{-150}
\end{pspicture}
}%
$$
and then erase all the end vertices of $\Delta$, and finally add some edges
labeled $\infty$.  When there is only one branch point there are no
$\infty$'s.  For example, the $Y_{555}$ diagram gives
$$
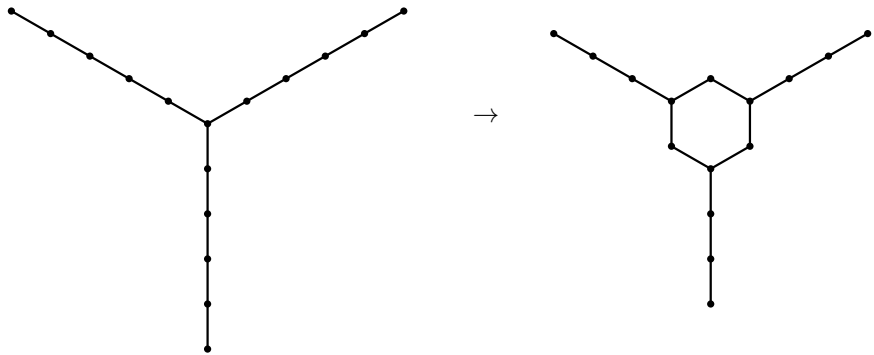
\raisebox{-.67\height}{%
\begin{pspicture}(-435,-500)(435,250)
\solid{0}{0}
\solid{0}{-100}
\solid{0}{-200}
\solid{0}{-300}
\solid{0}{-400}
\solid{0}{-500}
\solid{-87}{50}
\solid{-174}{100}
\solid{-261}{150}
\solid{-348}{200}
\solid{-435}{250}
\solid{87}{50}
\solid{174}{100}
\solid{261}{150}
\solid{348}{200}
\solid{435}{250}
\single{0}{0}{-435}{250}
\single{0}{0}{435}{250}
\single{0}{0}{0}{-500}
\end{pspicture}
}%
\qquad\to\qquad
\raisebox{-.67\height}{%
\begin{pspicture}(-335,-500)(335,250)
\solid{0}{-100}
\solid{0}{-200}
\solid{0}{-300}
\solid{0}{-400}
\solid{-87}{50}
\solid{-174}{100}
\solid{-261}{150}
\solid{-348}{200}
\solid{87}{50}
\solid{174}{100}
\solid{261}{150}
\solid{348}{200}
\single{-87}{50}{-348}{200}
\single{87}{50}{348}{200}
\single{0}{-100}{0}{-400}
\solid{0}{100}
\single{0}{100}{87}{50}
\single{0}{100}{-87}{50}
\solid{-87}{-50}
\single{-87}{-50}{-87}{50}
\single{-87}{-50}{0}{-100}
\solid{87}{-50}
\single{87}{-50}{87}{50}
\single{87}{-50}{0}{-100}
\end{pspicture}
}%
$$

We chose this example because it explains the appearance of the latter
figure in the ATLAS \cite{ATLAS} entry for the monster simple group
$M$, given the appearance of the former.  Namely, the bimonster
$(M\times M){:}2$ is described as a quotient of the $Y_{555}$
Coxeter group, and $M\times2$ as a quotient of the Coxeter group of
the second figure.  Given the first, one should expect the second,
because a $Y_{555}$ reflection maps to an involution in the bimonster
with centralizer $M\times2$.  (One can repeat the process, so the
reflection centralizer in the second diagram maps to the involution
centralizer $2B\times2\sset M\times2$ where $B$ is the baby monster.
Since the nonreflection part is now $\Z$, the $Y_{555}$ approach to
$M$  distinguishes a conjugacy class in $B$, up to inversion.  I
don't know what class this is or whether this approach offers any real
insight.)

Another example is $\Delta=E_8$, which we show in several steps to
illustrate the interaction between blowing up branch points and
erasing ends:
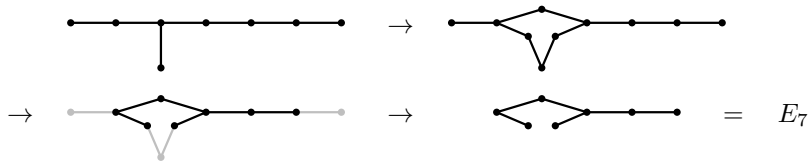
\begin{align*}
\raisebox{-.59\height}{%
\begin{pspicture}(-808,-108)(-192,38)
\solid{-200}{0}
\solid{-300}{0}
\solid{-400}{0}
\solid{-500}{0}
\solid{-600}{0}
\solid{-700}{0}
\solid{-800}{0}
\solid{-600}{-100}
\single{-200}{0}{-800}{0}
\single{-600}{0}{-600}{-100}
\end{pspicture}
}%
\quad&\to\quad
\raisebox{-.59\height}{%
\begin{pspicture}(-808,-108)(-192,38)
\solid{-200}{0}
\solid{-300}{0}
\solid{-400}{0}
\solid{-500}{0}
\solid{-700}{0}
\solid{-800}{0}
\solid{-600}{-100}
\single{-200}{0}{-500}{0}
\single{-700}{0}{-800}{0}
\solid{-600}{30}
\single{-600}{30}{-500}{0}
\single{-600}{30}{-700}{0}
\solid{-630}{-30}
\single{-630}{-30}{-700}{0}
\single{-630}{-30}{-600}{-100}
\solid{-570}{-30}
\single{-570}{-30}{-600}{-100}
\single{-570}{-30}{-500}{0}
\end{pspicture}
}%
\\
\noalign{\vskip5pt}
\to\quad
\raisebox{-.59\height}{%
\begin{pspicture}(-808,-108)(-192,38)
\psset{linecolor=lightgray}
\solid{-200}{0}
\single{-200}{0}{-300}{0}
\solid{-800}{0}
\single{-800}{0}{-700}{0}
\solid{-600}{-100}
\single{-630}{-30}{-600}{-100}
\single{-570}{-30}{-600}{-100}
\psset{linecolor=black}
\solid{-300}{0}
\solid{-400}{0}
\solid{-500}{0}
\solid{-700}{0}
\single{-300}{0}{-500}{0}
\solid{-600}{30}
\single{-600}{30}{-500}{0}
\single{-600}{30}{-700}{0}
\solid{-630}{-30}
\single{-630}{-30}{-700}{0}
\solid{-570}{-30}
\single{-570}{-30}{-500}{0}
\end{pspicture}
}%
\quad&\to\quad
\raisebox{-.59\height}{%
\begin{pspicture}(-808,-108)(-292,38)
\solid{-300}{0}
\solid{-400}{0}
\solid{-500}{0}
\solid{-700}{0}
\single{-300}{0}{-500}{0}
\solid{-600}{30}
\single{-600}{30}{-500}{0}
\single{-600}{30}{-700}{0}
\solid{-630}{-30}
\single{-630}{-30}{-700}{0}
\solid{-570}{-30}
\single{-570}{-30}{-500}{0}
\end{pspicture}
}%
\quad=\quad E_7
\end{align*}
The first step blows up the branch point, the second shows what will
be erased, and the third  actually erases it.  

An example
with an edge label $\infty$ is the reflection group
of the even unimodular lattice of signature $(17,1)$.  By
\cite{Vinberg-units} (see also
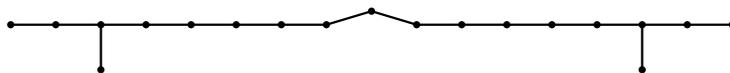
\cite{Conway}), 
$\Delta$ is 
$$
\raisebox{-.59\height}{%
\begin{pspicture}(-808,-108)(808,38)
\solid{-100}{0}
\solid{-200}{0}
\solid{-300}{0}
\solid{-400}{0}
\solid{-500}{0}
\solid{-600}{0}
\solid{-700}{0}
\solid{-800}{0}
\solid{-600}{-100}
\single{-100}{0}{-800}{0}
\single{-600}{0}{-600}{-100}
\solid{100}{0}
\solid{200}{0}
\solid{300}{0}
\solid{400}{0}
\solid{500}{0}
\solid{600}{0}
\solid{700}{0}
\solid{800}{0}
\solid{600}{-100}
\single{100}{0}{800}{0}
\single{600}{0}{600}{-100}
\solid{0}{30}
\single{0}{30}{100}{0}
\single{0}{30}{-100}{0}
\end{pspicture}
}%
$$
which after explosion and erasure yields
$$
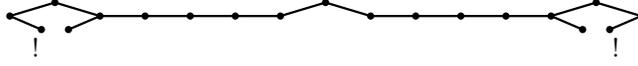
\raisebox{-.59\height}{%
\begin{pspicture}(-808,-108)(808,38)
\solid{-100}{0}
\solid{-200}{0}
\solid{-300}{0}
\solid{-400}{0}
\solid{-500}{0}
\solid{-700}{0}
\single{-100}{0}{-500}{0}
\solid{-600}{30}
\single{-600}{30}{-500}{0}
\single{-600}{30}{-700}{0}
\solid{-630}{-30}
\single{-630}{-30}{-700}{0}
\solid{-570}{-30}
\single{-570}{-30}{-500}{0}
\rput[tr](-635,-50){!}
\solid{100}{0}
\solid{200}{0}
\solid{300}{0}
\solid{400}{0}
\solid{500}{0}
\solid{700}{0}
\single{100}{0}{500}{0}
\solid{600}{30}
\single{600}{30}{500}{0}
\single{600}{30}{700}{0}
\solid{630}{-30}
\single{630}{-30}{700}{0}
\solid{570}{-30}
\single{570}{-30}{500}{0}
\rput[tl](635,-50){!}
\solid{0}{30}
\single{0}{30}{100}{0}
\single{0}{30}{-100}{0}
\end{pspicture}
}%
$$ Since the vertices marked ``!''\ correspond to $A_3$'s in $\Delta$ whose
convex hull is a $\tildeD_{16}$, they should be joined by an edge
labeled $\infty$.  Adjoining this edge completes the description of
$\DOmega$.  Remarks: (1) this is the reflection group of the even
lattice of signature $(16,1)$ and determinant~$-2$.  (2) Because it acts on
hyperbolic space $H^{16}$, it makes sense to ask whether this $\infty$
represents parallelism or ultraparallelism.  It represents
parallelism, because the corresponding infinite dihedral group lies in
the affine group $W(\tildeD_{16})$.

As a meatier example we treat a Coxeter group found by
Bugaenko \cite{Bugaenko}.  It acts cocompactly on $H^8$, and is the only
known cocompact example on any $H^{n\geq8}$.   Here $\Delta$ is
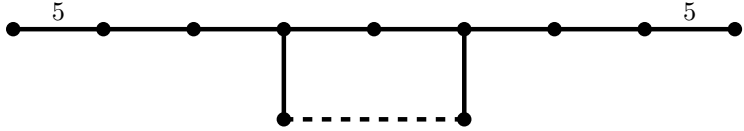
\psset{unit=\bigscaleunit}
\renewcommand{\noderadius}{\bigscalenoderadius}
\renewcommand{\singlewidth}{\bigscalesinglewidth}
$$
\begin{pspicture*}(-8,-108)(808,39)
\solid{0}{0}
\solid{100}{0}
\solid{200}{0}
\solid{300}{0}
\solid{400}{0}
\solid{500}{0}
\solid{600}{0}
\solid{700}{0}
\solid{800}{0}
\solid{300}{-100}
\solid{500}{-100}
\single{0}{0}{800}{0}
\single{300}{0}{300}{-100}
\single{500}{0}{500}{-100}
\dashed{300}{-100}{500}{-100}
\rput[b](50,10){$5$}
\rput[b](750,10){$5$}
\end{pspicture*}
$$
where the dashed line means an edge label~$\infty$.
The same argument as before shows that every tail class is
represented by a unique arrow from one vertex to another at
distance~$2$.  So we can name the $13$ generators of
$\WOmega$:
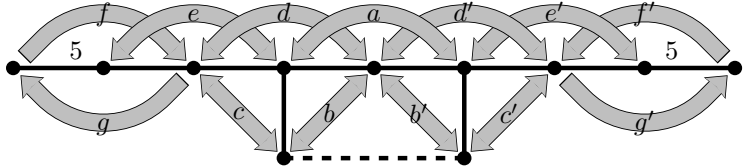
$$
\begin{pspicture}(-8,-108)(808,82)
\solid{0}{0}
\solid{100}{0}
\solid{200}{0}
\solid{300}{0}
\solid{400}{0}
\solid{500}{0}
\solid{600}{0}
\solid{700}{0}
\solid{800}{0}
\solid{300}{-100}
\solid{500}{-100}
\single{0}{0}{800}{0}
\single{300}{0}{300}{-100}
\single{500}{0}{500}{-100}
\dashed{300}{-100}{500}{-100}
\cwSE{100}{0}{$f$}
\ccwSW{700}{0}{$f'$}
\cwNW{100}{0}{$g$}
\ccwNE{700}{0}{$g'$}
\dbldown{200}{0}{$e$}
\dbldown{600}{0}{$e'$}
\dbldown{300}{0}{$d$}
\dbldown{500}{0}{$d'$}
\NWSE{200}{0}{$c$}
\SWNE{500}{-100}{$c'$}
\SWNE{300}{-100}{$b$}
\NWSE{400}{0}{$b'$}
\dbldown{400}{0}{$a$}
\rput[b](70,10){$5$}
\rput[b](730,10){$5$}
\end{pspicture}
$$
Note that $f$ and $g$ represent distinct arrow classes, because
only ``$A_3$ arrows'' are reversible.   Taking tails at
the left end of $\Delta$ shows that $f,e,d,c,a,b',d',e',g'$ are joined the
same way as their right endpoints are joined in $\Delta$.  Exchanging
primed and unprimed letters gives all joins among
$f',e',d',c',a,b,d,e,g$, so we know all the joins except those between
a member of $\{g,b,c',f'\}$ and a member of $\{g',b',c,f\}$.  By the
priming symmetry there are only 10 cases left to work out.  Taking
tails based at the middle vertex gives $b b'=\infty$ 
and $g g'=2$.  Taking tails based at the lower left vertex gives
$b c=b g'=c g=2$.  We have $f\!g=2$ by \eqref{eq-A5C3B}.
Finally, we have $b f=c c'=c f'=f\!f'=\infty$ because in none of these cases is there a
$3$-vertex spherical diagram containing representatives for both arrow
classes.  So $\DOmega$ is
$$
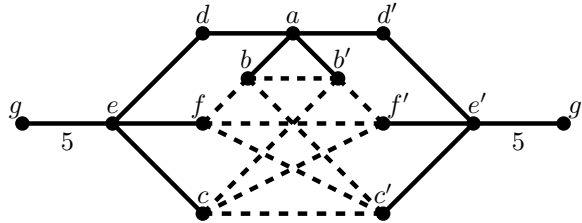
\newcommand{\aX}{0}  
\newcommand{\aY}{100}
\newcommand{\bX}{-50} 
\newcommand{\bY}{50}
\newcommand{\fX}{-100} 
\newcommand{\fY}{0}
\newcommand{\cX}{\fX} 
\newcommand{\cY}{-100}
\newcommand{\dX}{\fX} 
\newcommand{\dY}{\aY}
\newcommand{\eX}{-200} 
\newcommand{\eY}{\fY}
\newcommand{\gX}{-300} 
\newcommand{\gY}{\fY}
\newcommand{\fiveX}{-250} 
\newcommand{\fiveY}{-10}
\newcommand{\raisedtext}[3]{\rput[b](#1,#2){{\raise4pt\hbox{#3}}}}
\begin{pspicture}(\gX,\cY)(-\gX,140)
\solid{\aX}{\aY}
\single{\aX}{\aY}{\bX}{\bY}
\single{\aX}{\aY}{\dX}{\dY}
\single{\aX}{\aY}{-\bX}{\bY}
\single{\aX}{\aY}{-\dX}{\dY}
\raisedtext{\aX}{\aY}{$a$}
\solid{\bX}{\bY}
\solid{\cX}{\cY}
\solid{\dX}{\dY}
\solid{\eX}{\eY}
\solid{\fX}{\fY}
\solid{\gX}{\gY}
\raisedtext{\bX}{\bY}{$b\,$}
\raisedtext{\cX}{\cY}{$c$}
\raisedtext{\dX}{\dY}{$d$}
\raisedtext{\eX}{\eY}{$e$}
\raisedtext{\fX}{\fY}{$f\,\,$}
\raisedtext{\gX}{\gY}{$g\,\,\,$}
\single{\dX}{\dY}{\eX}{\eY}
\single{\fX}{\fY}{\eX}{\eY}
\single{\cX}{\cY}{\eX}{\eY}
\single{\gX}{\gY}{\eX}{\eY}
\rput[t](\fiveX,\fiveY){$5$}
\dashed{\bX}{\bY}{\fX}{\fY}
\solid{-\bX}{\bY}
\solid{-\cX}{\cY}
\solid{-\dX}{\dY}
\solid{-\eX}{\eY}
\solid{-\fX}{\fY}
\solid{-\gX}{\gY}
\raisedtext{-\bX}{\bY}{$\,\,b'$}
\raisedtext{-\cX}{\cY}{$c'$}
\raisedtext{-\dX}{\dY}{$\,\,d'$}
\raisedtext{-\eX}{\eY}{$\,\,e'$}
\raisedtext{-\fX}{\fY}{$\,\,\,\,\,\,\,f'$}
\raisedtext{-\gX}{\gY}{$\,\,\,\,\,\,\,g'$}
\single{-\dX}{\dY}{-\eX}{\eY}
\single{-\fX}{\fY}{-\eX}{\eY}
\single{-\cX}{\cY}{-\eX}{\eY}
\single{-\gX}{\gY}{-\eX}{\eY}
\rput[t](-\fiveX,\fiveY){$5$}
\dashed{-\bX}{\bY}{-\fX}{\fY}
\dashed{\bX}{\bY}{-\bX}{\bY}
\dashed{\bX}{\bY}{-\cX}{\cY}
\dashed{-\bX}{\bY}{\cX}{\cY}
\dashed{\cX}{\cY}{-\cX}{\cY}
\dashed{\fX}{\fY}{-\fX}{\fY}
\dashed{\cX}{\cY}{-\fX}{\fY}
\dashed{-\cX}{\cY}{\fX}{\fY}
\end{pspicture}
$$

We are grateful to one of the referees for suggesting the following
unpublished example of Howlett: the reflection centralizer when $\Delta$
is a $4$-cycle of edges labeled~$5$.  It turns out that $\WOmega$ has
countably many generators but no relations at all.
One can see this
as follows.  Here $\Dodd$ is a cycle, so $\tildeDodd$ is an infinite
chain.  An arrow has its tail at a point of this chain and its tip
in $\Delta$, evenly joined to the projection of the tail.  There is only
candidate for the tip, so there is one arrow for each point of the
chain.  None of the diagrams from table~\ref{tab-of-dihedral-angles} appear in $\Delta$, so
there is no fusion into arrow classes and there are no relations between the
arrows.

An interesting twist on this example is when $\Delta$ is a pentagon with
edges labeled~$5$.  For each pair of adjacent points of the chain
$\tildeDodd$ there is a unique point of $\Delta$ to which they both have
arrows.  These pairs are the tail classes, which are also the arrow
classes because there are no $A_3$ diagrams.  Given three consecutive
points of $\tildeDodd$, the first two give an arrow class and so do
the last two.   These arrow classes have representatives with the same tail (the
middle point), so they are joined in $\DOmega$ the same way their tips
are joined in $\Delta$, namely by an edge labeled~$5$.  
The end result is that
$\DOmega$ is an infinite chain with adjacent (resp.\ nonadjacent) nodes joined by edges
labeled~$5$ (resp.\ $\infty$).
Pursuing this to its logical
conclusion yields the following:

\begin{theorem}
Suppose $\Delta$ is a cycle of $n>3$ odd edges, with nodes labeled
$s_0,\dots,\discretionary{}{}{}s_{n-1}$ in cyclic order and subscripts read modulo~$n$.
Then
$\WOmega$ has generators $S_{i\in\Z}$ and relations

(i) $(S_i S_j)^{o(s_is_j)}=1$ if $|i-j|\leq n-4$

(ii) $(S_i S_{i+2-n})^2=1$ if $\{s_i,s_{i+1},s_{i+2}\}$ has type $H_3$

(iii) $S_i=S_{i+2-n}$ if $\{s_i,s_{i+1},s_{i+2}\}$ has type $A_3$

\noindent
A generator for $\GOmega$ acts by $S_i\mapsto S_{i+n}$.
\end{theorem}

\begin{remarks}
(1)
If $n\leq 3$ then $\WOmega$ is trivial.  (2) This is not quite a
Coxeter presentation since the third set of relations may identify
generators with each other.  After this identification one does obtain
a Coxeter presentation.  (3)  The affine diagram
$\tildeA_{n-1}$ is a cycle of $n$ edges labeled~$3$.  In this case the third set of
relations read $S_i=S_{i+2-n}$ for all $i$, so the centralizer has
$n-2$  generators.  The first set of relations shows that $\WOmega$
is the affine group $\tildeA_{n-3}$.   A generator for $\GOmega$
acts on this $\tildeA_{n-3}$ diagram by rotation by  two notches.
\end{remarks}

\begin{proof}[Proof sketch]
Call $\tildeDodd$'s nodes $t_{i\in\Z}$, where $t_i$ projects to $s_i$.
An arrow has the form $[t_i,s_j]$ where $i\not\cong j,j\pm1$
mod~$n$.  The tail classes are 
$$
S_i=\bigl\{[t_{i+2},s_i],[t_{i+3},s_i],\dots,[t_{i+n-2},s_i]\bigr\}
$$ 
for $i\in\Z$.
If $|i-i'|\leq n-4$ then $S_i$ and $S_{i'}$ contain arrows with the
same tail, hence the relation $(S_i S_{i'})^{o(s_i s_{i'})}=1$.  

If
$\{s_j,s_{j+1},s_{j+2}\}$ has type $H_3$, then 
\eqref{eq-A5C3B} says
that
$[t_{i+2},s_i]$ 
and
$[t_i,s_{i+2}]$ 
commute, whenever $i\in\Z$ reduces to $j$ modulo $n$.
Since $[t_{i+2},s_i]\in S_i$ and $[t_i,s_{i+2}]\in S_{i+2-n}$, this
gives
$(S_i S_{i+2-n})^2=1$.  
If
$\{s_j,s_{j+1},s_{j+2}\}$ has type $A_3$ then 
the same argument using \eqref{eq-A3C3B} gives $S_i=S_{i+2-n}$ instead.
The cases \eqref{eq-A4B3C}, \eqref{eq-AevenB2C} and \eqref{eq-B4A3C}
from table~\ref{tab-of-dihedral-angles} are irrelevant.
\end{proof}

As a final example we mention Vinberg's diagrams
\cite[p.~34]{Vinberg-units} for the reflection groups of the odd
unimodular Lorentzian lattices.  If $s$ is a node of one of these
diagrams corresponding to a norm~$1$ root then $\WOmega$ is the
reflection group of its orthogonal complement, which is unimodular
Lorentzian of one dimension less.  One possibility for $s$ (usually
the only one) leads to $\WOmega$ being the previous entry in Vinberg's
table.  By taking centralizers of nodes corresponding to norm~$2$
roots one could instead rederive the
diagrams for the reflection groups of the odd bimodular Lorentzian
lattices 
\cite[p.~32]{Vinberg-units}.

\end{document}